\documentclass[aap,MSNbibl,nameyear,dvips]{arximspdf}
\usepackage{mathbh}
\usepackage{algorithmic, algorithm}
\usepackage{graphicx}

\doi{10.1214/12-AAP913} 
\volume{24}
\issue{1}
\pubyear{2014}
\firstpage{34}
\lastpage{53}

\makeatletter

\newcommand{\rrvert}{\vert}
\newcommand{\llvert}{\vert}
\newtheorem{theorem}{Theorem}
\newtheorem{lemma}[theorem]{Lemma}
\newproclaim{assumption}{Assumption}
\newtheorem{corollary}[theorem]{Corollary}

\newcommand{\Ind}{\mathbh{1}}
\makeatother

\begin{document}
\begin{frontmatter}

\title{The Wang--Landau algorithm reaches the flat histogram criterion in finite time}
\runtitle{Wang--Landau Flat Histogram in finite time}

\begin{aug}
\author[A]{\fnms{Pierre E.} \snm{Jacob}\thanksref{nus}\ead[label=e1]{stajpe@nus.edu.sg}}
\and
\author[B]{\fnms{Robin J.} \snm{Ryder}\corref{}\thanksref{gis}\ead[label=e2]{ryder@ceremade.dauphine.fr}}
\thankstext{nus}{Supported in part by AXA research.}
\thankstext{gis}{Supported in part by a GIS Sciences de la D\'ecision postdoc scholarship.}
\runauthor{P.~E. Jacob and R.~J. Ryder}
\affiliation{National University of Singapore and Universit\'e Paris Dauphine}
\address[A]{Department of Statistics \&\\
\quad Applied Probability\\
National University of Singapore\\
Singapore\\
\printead{e1}}
\address[B]{Centre de Recherche en \'Economie\\
\quad et Statistique and CEREMADE\\
Universit\'e Paris Dauphine\\
Place du Mar\'echal de Lattre de Tassigny\\
75016 Paris\\
France\\
\printead{e2}}
\end{aug}

\received{\smonth{10} \syear{2011}}
\revised{\smonth{12} \syear{2012}}

%
\begin{abstract}
The Wang--Landau algorithm aims at sampling from a probability
distribution, while penalizing some regions of the state space and
favoring others. It is widely used, but its convergence properties are
still unknown. We show that for some variations of the algorithm, the
Wang--Landau algorithm reaches the so-called flat histogram criterion
in finite time, and that this criterion can be never reached for other
variations. The arguments are shown in a simple context---compact
spaces, density functions bounded from both sides---for the sake of
clarity, and could be extended to more general contexts.
\end{abstract}

%
\begin{keyword}[class=AMS]
\kwd[Primary ]{65C05}
\kwd{60J22}
\kwd[; secondary ]{60J05}
\end{keyword}
\begin{keyword}
\kwd{Markov chain}
\kwd{Monte Carlo}
\kwd{Wang--Landau algorithm}
\kwd{flat histogram}
\end{keyword}

\end{frontmatter}

\section{Introduction and notation}\label{sec1}

Consider the problem of sampling from a probability distribution $\pi$ defined
on a measure space $(\mathcal{X}, \Sigma, \mu)$. We suppose that we
can compute
the probability density function of $\pi$ at any point \mbox{$x\in\mathcal
{X}$}, up to
a multiplicative constant. Given a proposal kernel $Q(\cdot,\cdot)$
we define
a Metropolis--Hastings (MH) [\citet{Hastings,Tierney}] transition kernel
targeting $\pi$, denoted by $K(\cdot, \cdot)$, as follows:
\[
\forall x,y \in\mathcal{X}\qquad K(x, y) = Q(x, y) \rho(x, y) + \delta_x(y)
\bigl(1 - r(x) \bigr)
\]
with $\rho(x, y)$ defined by $\rho(x, y):= 1 \wedge
\frac{\pi(y)}{\pi(x)}\frac{Q(y,x)}{Q(x,y)}$ and $r(x)$ defined by
\[
r(x):= \int_\mathcal{X} \rho(x,y) Q(x, y) \,dy.
\]
Here the delta function $\delta_a(b)$ takes value $1$ when $a = b$ and $0$
otherwise. Under some conditions on the proposal $Q$ and the target
$\pi
$, the
MH kernel defines an algorithm to generate a Markov chain with stationary
distribution~$\pi$ [\citet{Robert2004a}].

Let us consider a partition of the state space $\mathcal{X}$ into $d$ disjoint
sets $\mathcal{X}_1, \ldots, \mathcal{X}_d$,
\[
\mathcal{X} = \bigcup_{i=1}^d
\mathcal{X}_i.
\]
If we have a sample $X_1, \ldots, X_t$ independent and identically distributed
from $\pi$, then for any $i\in[1,d]$,
\[
\frac{1}{t}\sum_{n=1}^t
\Ind_{\mathcal{X}_i}(X_n) \mathop{\longrightarrow}^{\mathbb{P}}_{t\to
\infty}
\int_{\mathcal{X}_i}\pi(x)\,dx =: \psi_i,
\]
where we denote by $\Ind_{\mathcal{X}_i}(x)$ the indicator function
that is
equal to $1$ when $x\in\mathcal{X}_i$ and $0$ otherwise. Similar
convergence is
obtained when $X_1, \ldots, X_t$ is an ergodic chain such as the one generated
by the MH algorithm. The purpose of the
Wang--Landau algorithm [Wang and Landau (\citeyear{Wang2001a,Wang2001b}), \citet{Liang2005a,Atchade2010a}] is
to obtain a sample:
\begin{itemize}
\item
such that for any $i\in[1,d]$, the subsample
\[
\bigl\{X_n \mbox{ for } n \in[1,t] \mbox{ s.t. } X_n
\in\mathcal{X}_i \bigr\}
\]
is distributed according to the restriction of $\pi$ to
$\mathcal{X}_i$, and
\item such that for any $i\in[1,d]$,
\[
\frac{1}{t}\sum_{n=1}^t
\Ind_{\mathcal{X}_i}(X_n) \mathop{\longrightarrow}^{\mathbb{P}}_{t\to\infty}
\phi_i,
\]
where $\phi= (\phi_1, \ldots, \phi_d)$ is chosen by the user, and
could be any vector in $]0,1[^d$ such that $\sum_{i=1}^d \phi_i = 1$.
\end{itemize}
A typical use of this algorithm is to sample from multimodal
distributions, by
penalizing already-visited regions and favoring the exploration of regions
between modes, in an attempt to recover all the modes. In many
applications, $\phi$
is set to $\forall i, \phi_i=1/d$; however, other choices are
possible, as exemplified
by the adaptive binning scheme of \citet{bornn2011adaptive}.

This algorithm, in the class of Markov chain Monte Carlo (MCMC) algorithms
[\citet{Robert2004a}], therefore allows us to learn about $\pi$ while
``forcing''
the proportions of visits $\phi_i$ of the generated chain to any of
the sets
$\mathcal{X}_i$, which are typically also chosen by the user. The vector
$\phi_1, \ldots, \phi_d$ might be referred to as the ``desired frequencies,''
and the sets $\mathcal{X}_i$ are called the ``bins.'' In a typical situation,
the mass of $\pi$ over bin $\mathcal{X}_i$, which we denote by $\psi
_i$, is
unknown, and hence one cannot easily guess how much to ``penalize'' or to
``favor'' a bin $\mathcal{X}_i$ in order to obtain the desired frequency
$\phi_i$. The Wang--Landau algorithm introduces a vector $\theta_t =
(\theta_t(1), \ldots, \theta_t(d))$, referred to as ``penalties'' at
time $t$,
which is updated at every iteration $t$, and which acts like an approximation
of the ratios $\psi_1 / \phi_1, \ldots, \psi_d / \phi_d$, up to
a multiplicative constant.\vadjust{\goodbreak}

For a distribution $\pi$ and a vector of penalties $\theta= (\theta(1),
\ldots, \theta(d))$, we define the penalized distribution $\pi
_\theta$,
\[
\pi_{\theta}(x) \propto\pi(x) \times\sum_{i=1}^{d}
\frac{\Ind_{\mathcal{X}_i}(x)}{\theta(i)}.
\]
To be more concise we define a function $J\dvtx\mathcal{X}\mapsto\{
1,\ldots,d\}$
that takes a state $x \in\mathcal{X}$ and returns the index $i$ of
the bin
$\mathcal{X}_i$ such that $x\in\mathcal{X}_i$. We can now write
\(\pi_\theta(x) \propto\pi(x)/\theta(J(x))\). We will denote by
$K_{\theta}$
the MH kernel targeting~$\pi_\theta$.

The Wang--Landau algorithm, described in the next section, alternates between
generating a sample by targeting $\pi_\theta$ using $K_\theta$, and updating
$\theta$ using the generated sample. In this sense it is an adaptive MCMC
algorithm (past samples are used to update the kernel at a given iteration),
using an auxiliary chain $(\theta_t)$, and therefore the behavior of the
sample is not obvious.

The Wang--Landau algorithm is widely used in the Physics community
[\citet{silva2006wang}, \citet{malakis2006monte}, \citet{cunha2006wang}]. In particular, many
practitioners use flavors of the algorithm with a ``flat histogram''
criterion. However, its convergence properties are still partially
unknown. We
show that this criterion is reached in finite time for some variations
of the algorithm. This result is all that was
missing to apply results on adaptive algorithms with diminishing adaptation
[\citet{FortMoulinesCLT}].

In Section \ref{secWLdefinitions}, we define variations of the Wang--Landau
algorithm. We then introduce ratios of penalties and argue for their
convenience in studying the properties of the algorithm.
We prove in Sections \ref{secdequals2} and \ref{secproofanyd} that
under certain
conditions, the flat histogram criterion is met in finite time, for the cases
$d = 2$ and $d > 2$, respectively. The result is
illustrated in Section \ref{sectoyexample}, and in Section
\ref{secdiscussion}, we hint at how our assumptions might be relaxed.

\section{Wang--Landau algorithms: Different flavors}\label{secWLdefinitions}

There are several versions of the Wang--Landau algorithm. We describe the
general version introduced by \citet{Atchade2010a}, both in its deterministic
form and with a stochastic schedule.

\subsection{A first version with deterministic schedule}

Let $(\gamma_t)_{t\in\mathbb{N}}$ (referred to as a schedule or a
temperature)
be a sequence of positive real numbers such that
\[
\cases{ \displaystyle\sum_{t \geq0}\gamma_t =
\infty,\vspace*{2pt}
\cr
\displaystyle\sum_{t \geq0}
\gamma_t^2 < \infty.} %
\]
A typical choice is $\gamma_t:= t^{-\alpha}$ with $\alpha\in\,]0.5,
1[$. The
Wang--Landau algorithm is described in pseudo-code in Algorithm \ref{algoWL}.
In this form, the schedule $\gamma_t$ decreases at each iteration, and is
therefore called ``deterministic.''
\begin{algorithm}[t]
\caption{Wang--Landau with deterministic schedule\label{algoWL}.}
\begin{algorithmic}[1]
\STATE Init $\forall i\in\{ 1, \ldots, d\}$ set $\theta_0(i)
\leftarrow1/d$.
\STATE Init $X_0 \in\mathcal{X}$.
\FOR{$t=1$ to $T$}
\STATE Sample $X_t$ from $K_{\theta_{t-1}}(X_{t-1},\cdot)$, MH
kernel targeting
$\pi_{\theta_{t-1}}$.
\STATE\label{algoWLupdate} Update the penalties: $\log\theta_{t}(i)
\leftarrow
\log\theta_{t-1}(i) +
f(\Ind_{\mathcal{X}_i}(X_t),\phi_i, \gamma_t)$.
\ENDFOR
\end{algorithmic}
\end{algorithm}

Step \ref{algoWLupdate} of Algorithm \ref{algoWL} updates the
penalties from
$\theta_{t-1}$ to $\theta_t$, by increasing it if the corresponding
bin has
been visited by the chain at the current iteration, and by decreasing it
otherwise. This rationale seems natural; however, we did not find any article
arguing for a particular choice of update, among the infinite number of updates
that would also follow the same rationale. In other words, it is not obvious
how to choose the function $f$, except that it should be such that it is
positive when $X_t \in\mathcal{X}_i$ and such that it is closer to
$0$ when
$\gamma_t$ decreases, to ensure that the penalties converge. Some
practitioners use the following update:
%
%
\begin{equation}
\label{eqrightupdate} \log\theta_{t}(i) \leftarrow\log
\theta_{t-1}(i) + \gamma_t \bigl(\Ind_{\mathcal{X}_i}(X_t)
- \phi_i \bigr)
\end{equation}
while others use
%
%
\begin{equation}
\label{eqwrongupdate} \log\theta_{t}(i) \leftarrow\log
\theta_{t-1}(i) + \log\bigl[1 + \gamma_t \bigl(
\Ind_{\mathcal{X}_i}(X_t) - \phi_i \bigr) \bigr].
\end{equation}
Since $\gamma_t$ converges to $0$ when $t$ increases, and since update
(\ref{eqrightupdate}) is the first-order Taylor expansion of update
(\ref{eqwrongupdate}), one legitimately expects both updates to result in
similar performance in practice. We shall see in Section \ref
{secdequals2} that
this is not necessarily the case.

Some convergence results have been proven about Algorithm \ref
{algoWL}: the
deterministic schedule ensures that $\theta_t$ changes less and less
along the
iterations of the algorithm, and consequently the kernels $K_{\theta
_t}$ change
less and less as well. The study of the algorithm hence falls into the
realm of
adaptive MCMC where the \emph{diminishing adaptation} condition holds
[\citet{Andrieu2008b,Atchade2009c,FortMoulinesCLT}], although it is
original in
the sense that the target distribution, ($\pi_{\theta_t}$) is adaptive,
but not
necessarily the proposal distribution $Q$. See also the literature on
stochastic approximation, \citet{Andrieu2006b}.

In this article we are especially interested in a more sophisticated
version of
the Wang--Landau algorithm that uses a stochastic schedule, and for
which, as we
shall see in the following, the two updates result in different performance.

\subsection{A sophisticated version with stochastic schedule}

A remarkable improvement has been made over Algorithm \ref{algoWL}:
the use
of a ``flat histogram'' (FH) criterion to decrease the schedule only at certain
random times. Let us introduce $\nu_t(i)$, the number of generated
points at
iteration $t$ that are in~$\mathcal{X}_i$,
\[
\nu_t(i):= \sum_{n=1}^t
\Ind_{\mathcal{X}_i}(X_n).
\]
For some predefined precision threshold $c$, we
say that FH is met at iteration $t$ if
\[
\max_{i\in\{1,\ldots,d\}} \biggl\llvert\frac{\nu_t(i)}{t} -
\phi_i \biggr\rrvert< c.
\]
Intuitively, this
criterion is met if the observed proportion of visits to each bin is
not far
from $\phi$, the desired proportion. The name ``flat histogram'' comes from
the observation that if the desired proportions are all equal to $1/d$, this
criterion is verified when the histogram of visits is approximately
flat. The
threshold $c$ could possibly decrease along the iterations to get an always
finer precision.

The Wang--Landau with flat histogram (Algorithm \ref{algoWLFH}) is
similar to
the previous algorithm, with a single difference: the schedule $\gamma
$ does
not decrease at each step anymore, but only when FH is met. To know whether
it is met or not, a counter $\nu_t$ of visits to each bin is updated
at each
iteration, and when FH is met, the schedule decreases and the counter is
reset to $0$.

\begin{algorithm}[t]
\caption{Wang--Landau with flat histogram \label{algoWLFH}}
\begin{algorithmic}[1]
\STATE Init $\forall i\in\{ 1, \ldots, d\}$ set $\theta_0(i)
\leftarrow1/d$.
\STATE Init $X_0 \in\mathcal{X}$.
\STATE Init $\kappa= 0$, the number of FH criteria already reached.
\STATE Init the counter $\forall i\in\{ 1, \ldots, d\}\ \nu_1(i)
\leftarrow0$
\FOR{$t=1$ to $T$}
\STATE Sample $X_t$ from $K_{\theta_{t-1}}(X_{t-1},\cdot)$ targeting
$\pi_{\theta_{t-1}}$.
\STATE Update $\nu_t$: $\forall i\in\{ 1, \ldots, d\}\ \nu_t(i)
\leftarrow\nu_{t-1}(i)+\Ind_{\mathcal X_i}(X_t)$
\STATE Check whether FH is met.
\IF{FH is met}
\STATE$\kappa\leftarrow\kappa+ 1$
\STATE$\forall i\in\{ 1, \ldots, d\}\ \nu_t(i)\leftarrow0$
\ENDIF
\STATE\label{algoWLFHupdate} Update the bias:
$\log\theta_t(i) \leftarrow\log\theta_{t-1}(i) +
f(\Ind_{\mathcal{X}_i}(X_t),\phi_i, \gamma_\kappa)$.
\ENDFOR
\end{algorithmic}
\end{algorithm}

Note the difference between Algorithms \ref{algoWL} and \ref{algoWLFH}:
$\gamma$ is indexed by $\kappa$ instead of~$t$, and $\kappa$ is a random
variable. As with Algorithm \ref{algoWL}, the update of penalties (step
\ref{algoWLFHupdate} of Algorithm \ref{algoWLFH}) can be either update
(\ref{eqrightupdate}) or update (\ref{eqwrongupdate}), or possibly something
else.
Interestingly in this case, it is not obvious anymore that both updates will
give similar results. Indeed, for $\gamma_\kappa$ to go to $0$, we
need FH
to be reached in finite time, so that $\kappa$ regularly increases.

This flavor of the Wang--Landau algorithm is widely used in the Physics
literature [\citet{cunha2006wang,silva2006wang,malakis2006monte,ngo2008phase}].

Our contribution is to show in a simple context that update
(\ref{eqrightupdate}) is such that FH is met in finite time, while
(\ref{eqwrongupdate}) is not so. Hence only using update
(\ref{eqrightupdate}) can one expect the convergence properties of Algorithm~\ref{algoWL}
to still hold for Algorithm~\ref{algoWLFH}, since if~FH is met
in finite time, a sort of \emph{diminishing adaptation} condition
would still
hold.

To underline the difficulty of knowing whether FH is met in finite
time or
not, let us recall that between two FH occurrences, the schedule is constant
(equal to some $\gamma_\kappa> 0$); hence the penalties $(\theta_t)$ change
at a constant scale and \emph{diminishing adaptation} does not
directly hold.
Other adaptive algorithms share this lack of \emph{diminishing
adaptation}, as, for example, the accelerated stochastic approximation algorithm
[\citet{Kesten1958}], in which the adaptation of some process diminishes
only if its increments change sign. In our case, FH will be reached
if the
chain $(X_t)$ lands with frequency $\phi_i$ in each bin
$\mathcal{X}_i$; see Corollary \ref{corolFH}.

Note that in the implementation of the algorithm, the penalties
$\theta_t$ need only be defined up to a normalizing constant, since
they only
appear in ratios of the form $\theta_t(i)/\theta_t(j)$. We therefore
introduce the following notation:
\[
\forall i,j\in\{1,\ldots,d\} \mbox{ such that } i \neq j \qquad Z_t^{(i,j)}
= \log\frac{\theta_t(i)}{\theta_t(j)},
\]
and we note $Z_t$ the collection of all the $Z_t^{(i,j)}$.
Some intuition behind the study of such ratios comes from considering
update (\ref{eqrightupdate}). With this
update, assume that for each $i$, $\mathbb{E}[\Ind_{\mathcal
X_i}(X_t)]=\phi_i$. Then we could easily check that
for each pair $(i, j)$, $\mathbb
{E}[Z_t^{(i,j)}|Z_{t-1}^{(i,j)}]=Z_{t-1}^{(i,j)}$,
so this process would be constant on average.
The remainder of this paper hinges on two facts:
that we can control $(Z_t)$ in the sense that $Z_t^{(i,j)}/t \to0$,
and that
if we control $(Z_t)$, then we control the frequencies of visits $(\nu_t/t)$.

More generally, notice that with fixed
$\gamma$, the pair $(X_t, Z_t)$ forms a homogeneous Markov chain. We
would like to prove that its
proportion of visits to the set $\mathcal X_i \times\mathbb R^{d(d-1)}$
converges to some value in $[0,1]$; one way to prove this is to show
that the chain is irreducible.
We would then need to check that the limit
is indeed the desired frequency $\phi_i$ for all $i$. Unfortunately, properties
of the joint chain $(X_t, Z_t)$ are difficult to establish due to the
complexity of its transition kernel. Finding a so-called drift function
for the joint
Markov chain is also typically difficult. In general, we are not able
to show
that the chain is irreducible. In Section \ref{secdequals2}, we prove
directly that $Z_t^{1,2}/t\to0$ in the special case $d=2$, under some
assumptions. In Section \ref{secproofanyd}, we make more restrictive
assumptions which imply irreducibility. In both cases, we show the implication
of this convergence on the frequencies of visits.

\section{Proof when $d=2$}\label{secdequals2}

In the following we consider a simple context with only two bins: $d = 2$
and $X_t$ can therefore only be either in $\mathcal{X}_1$ or in
$\mathcal{X}_2$.
Suppose the current schedule is at $\gamma> 0$, and we want to know whether
FH is going to be met in finite time (hence $\gamma$ is fixed here).
To simplify notation,
in this section we note
\[
Z_t = Z_t^{(1,2)} = \log\theta_t(1) -
\log\theta_t(2).
\]

Using the definition of the penalties
$(\theta_t)$ and of the counts $(\nu_t)$, we obtain
\begin{eqnarray*}
Z_t& =& Z_0 + \bigl[\nu_t(1) f(1,
\phi_1, \gamma) + \bigl(t - \nu_t(1) \bigr) f(0,
\phi_1, \gamma) \bigr]
\\
&&{}- \bigl[\nu_t(2) f(1, \phi_2, \gamma) + \bigl(t -
\nu_t(2) \bigr) f(0, \phi_2, \gamma) \bigr]
\\
&=& Z_0 + \nu_t(1) \bigl[f(1, \phi_1,
\gamma) - f(0, \phi_1, \gamma) \bigr] + t f(0, \phi_1,
\gamma)
\\
&&{}- \bigl[ \bigl(t - \nu_t(1) \bigr) f(1, \phi_2, \gamma)
+ \nu_t(1) f(0, \phi_2, \gamma) \bigr]
\\
&=& Z_0 + \nu_t(1) \bigl[f(1, \phi_1,
\gamma) - f(0, \phi_1, \gamma) + f(1, \phi_2, \gamma) -
f(0, \phi_2, \gamma) \bigr]
\\
&&{}+ t \bigl(f(0, \phi_1, \gamma) - f(1, \phi_2, \gamma)
\bigr).
\end{eqnarray*}

If we prove that $Z_t / t$ goes to $0$ (e.g., in mean), this
will imply
the following convergence of the proportion of visits:
\[
\frac{\nu_t(1)}{t} \mathop{\longrightarrow}_{t\to\infty} \frac{f(1, \phi_2,
\gamma) - f(0, \phi_1,
\gamma)}{
f(1, \phi_1, \gamma) - f(0, \phi_1, \gamma) +
f(1, \phi_2, \gamma) - f(0, \phi_2, \gamma)
}
\]
(also in mean). Since we want FH to be reached in finite time for any
precision threshold $c>0$, we need the proportions of visits to
$\mathcal{X}_i$
to converge to $\phi_i$. Hence we want
%
%
\begin{equation}
\label{equpdateconstraint} \frac{f(1, \phi_2, \gamma) - f(0, \phi_1,
\gamma)}{
f(1, \phi_1, \gamma) - f(0, \phi_1, \gamma) +
f(1, \phi_2, \gamma) - f(0, \phi_2, \gamma)
} = \phi_1.
\end{equation}
Using the specific forms of $f(\Ind_{\mathcal{X}_i}(X_t), \phi_i,
\gamma
)$ for
both updates, we can easily see that:
\begin{itemize}
\item update (\ref{eqrightupdate}) satisfies equation
(\ref{equpdateconstraint}) for any $\phi$ and $\gamma$;
\item in general, update (\ref{eqwrongupdate}) does not satisfy equation
(\ref{equpdateconstraint}), except in the special case where
$\phi_1=\phi_2=1/2$.
\end{itemize}

Note that the second point states that the proportions of visits
converge to some vector
$\tilde{\phi}$ different than $\phi$. The vector $\tilde{\phi}$
can be
expressed as a function
of $\phi\dvtx\tilde{\phi} = g(\phi)$. By numerically or analytically
inverting this function $g$,
one can plug $g^{-1}(\phi)$ as an algorithmic parameter, so that
the limiting proportions of\vadjust{\goodbreak} visits converge to $g(g^{-1}(\phi)) = \phi
$. Hence
one can use update (\ref{eqwrongupdate}) and get the desired
proportions of visits
by plugging $g^{-1}(\phi)$ instead of $\phi$ in the update.

The rest of the paper is devoted to the proof that $Z_t/t$ goes to $0$ under
some assumptions. More formally, we state in Theorem \ref{thmconv}
what we
shall prove in the remainder of this section.
This theorem holds for both updates.

%
\begin{theorem}\label{thmconv}
Consider the sequence of penalties $(\theta_t)$ introduced in Algorithm~\ref{algoWLFH}. We define
\[
Z_t = \log\theta_t(1) - \log\theta_t(2).
\]
Then
\[
\frac{Z_t}{t} \mathop{\longrightarrow}^{L_1}_{t\to\infty} 0.
\]
\end{theorem}

As a consequence, the long run proportion of visits to each bin converges
to the desired frequency $\phi$ for update (\ref{eqrightupdate}),
and not necessarily for update (\ref{eqwrongupdate}).
Corollary \ref{corolFH} clarifies the consequence of Theorem \ref{thmconv}
on the validity of Algorithm \ref{algoWLFH}.
%
%
\begin{corollary}\label{corolFH}
When the proportions of visits converge in mean to the desired
proportions, the
flat histogram criterion is reached in finite time for any precision threshold
$c$.
\end{corollary}

We already made the simplification of considering the simple case $d = 2$.
We make the following assumptions:
%
%
\begin{assumption}\label{asemptybins}
The bins are not empty with respect to $\mu$ and $\pi$,
\[
\forall i \in\{1, 2\} \qquad \mu(\mathcal{X}_i) > 0\quad \mbox{and}\quad \pi(
\mathcal{X}_i) > 0.
\]
\end{assumption}
%
%
\begin{assumption}\label{ascompact}
The state space $\mathcal{X}$ is compact.
\end{assumption}
%
%
\begin{assumption}\label{asproposal}
The proposition distribution $Q(x,y)$ is such that
\[
\exists q_{\mathrm{min}} > 0\ \forall x \in\mathcal{X}\ \forall y \in\mathcal{X}\qquad
Q(x,y) > q_{\mathrm{min}}.
\]
\end{assumption}
%
%
\begin{assumption}\label{asmhratio}
The MH acceptance ratio is bounded from both sides
\[
\exists m > 0\ \exists M > 0\ \forall x \in\mathcal{X}\ \forall y \in
\mathcal{X}\qquad  m
< \frac{\pi(y)}{\pi(x)}\frac{Q(y,x)}{Q(x,y)}< M.
\]
\end{assumption}

Assumption \ref{asemptybins} guarantees that the bins are well
designed, and
if it was not verified, the algorithm would never reach FH,
regardless of the
other assumptions. Assumptions \ref{ascompact}--\ref{asmhratio} are, for
example, verified by a Gaussian random walk proposal over a compact
space, where
there is a lower bound on $\pi$. We believe that these assumptions can be
relaxed to cover the most general Wang--Landau\vadjust{\goodbreak} algorithm. Making these four
assumptions allows us to propose a clearer proof, and we propose hints
on how to
relax them in Section~\ref{secdiscussion}.

We denote by $U_t$ the increment of $Z_t$, such that for any $t$,
\[
Z_{t+1} = Z_t + U_t = Z_t + f
\bigl( \Ind_{\mathcal{X}_1}(X_{t}), \phi_1, \gamma\bigr) -
f \bigl( \Ind_{\mathcal{X}_2}(X_t), \phi_2, \gamma\bigr).
\]
Here with only two bins, the
increments $U_t$ can take two different values, $+a$ or $-b$, for some $a>0$
and $b>0$ that depend on $\phi$ and $\gamma$. For example, with update~(\ref{eqrightupdate}),
\[
\cases{ a = 2\gamma(1 - \phi_1) > 0,\vspace*{2pt}
\cr
b
= 2\gamma\phi_1 > 0,} %
\]
whereas with update (\ref{eqwrongupdate}),
\[
\cases{ \displaystyle a = \log\biggl(\frac{1 + \gamma(1 - \phi_1)}{1 - \gamma
(1 - \phi
_1)} \biggr)> 0,\vspace*{2pt}
\cr
\displaystyle b= \log\biggl(\frac{1 + \gamma\phi_1}{1 - \gamma\phi_1} \biggr
)>0, } %
\]
and in both cases, if $X_t\in\mathcal{X}_1$, then $U_t = +a$, otherwise
$U_t = -b$.

We want to prove that $Z_t / t$ goes to $0$,\vspace*{1pt} and we are going to prove a
stronger result that states, in words, that when $Z_t$ leaves a fixed interval
$[\bar{Z}^{\mathrm{lo}}, \bar{Z}^{\mathrm{hi}}]$, it returns to it in a finite time.

\subsection{Behavior of $(Z_t)$ outside an interval}
First, Lemma \ref{lemma1} states that if $Z_t$ goes above a value
$\bar{Z}^{\mathrm{hi}}$, it has a strictly positive probability of starting to decrease,
and that when that happens, it keeps on decreasing with a high probability.

%
\begin{lemma}\label{lemma1}
With the introduced processes $Z_t$ and $U_t$,
there exists $\epsilon> 0$ such that for all $\eta> 0$, there
exists $\bar{Z}^{\mathrm{hi}}$ such that, if $Z_t\geq\bar{Z}^{\mathrm{hi}}$, we have the
following two inequalities:
\begin{eqnarray*}
P[U_{t+1}=-b|U_t=+a,Z_t]&>&\epsilon,
\\
P[U_{t+1}=-b|U_t=-b,Z_t]&>&1-\eta.
\end{eqnarray*}
\end{lemma}

\begin{pf}
We start with the first inequality. Let $q_{\mathrm{min}}$ be like in Assumption~\ref{asproposal}.

In terms of events $\{U_t = +a\}$ is equivalent to $\{X_t \in\mathcal
{X}_1\}$,
by definition. If $X_t \in\mathcal{X}_1$ and $\pi(X_t) > 0$, then
\begin{eqnarray*}
K_{\theta_t}(X_t, \mathcal{X}_2) &=& \int
_{\mathcal{X}_2} K_{\theta_t}(X_t, y) \,dy
\\
&=& \int_{\mathcal{X}_2} Q(X_t, y) \rho_{\theta_t}(X_t,
y) \,dy
\\
&=& \int_{\mathcal{X}_2} Q(X_t, y) \biggl( 1 \wedge
\frac{\pi(y)}{\pi(X_t)}\frac{Q(y,X_t)}{Q(X_t,y)}\frac{\theta_t(J(X_t))}{
\theta_t(J(y)) } \biggr)\,dy
\\
&= &\int_{\mathcal{X}_2} Q(X_t, y) \biggl( 1 \wedge
\frac{\pi(y)}{\pi(X_t)}\frac{Q(y,X_t)}{Q(X_t,y)}e^{Z_t} \biggr)\,dy.
\end{eqnarray*}

Using Assumption \ref{asmhratio}, $\frac{\pi(y)}{\pi(x)}\frac
{Q(y,x)}{Q(x,y)}$
is bounded from below, hence there exists $K_1$ such that
\[
\forall k \geq K_1\ \forall x,y\in\mathcal{X}\qquad \frac{\pi(y)}{\pi(x)}
\frac{Q(y,x)}{Q(x,y)}e^{k} \geq1.
\]
If $Z_t \geq K_1$ and $X_t\in\mathcal{X}_1$, then
\[
K_{\theta_t}(X_t, \mathcal{X}_2) = \int
_{\mathcal{X}_2} Q(X_t, y) \,dy > q_{\mathrm{min}} \mu(
\mathcal{X}_2).
\]

Hence if $Z_t\geq K_1$,
\begin{eqnarray*}
P[U_{t+1}=-b|U_t=+a,Z_t] &=& P[X_{t+1}
\in\mathcal{X}_2 | X_t \in\mathcal{X}_1,
Z_t]
\\
&>& q_{\mathrm{min}} \mu(\mathcal{X}_2).
\end{eqnarray*}

We now prove the second inequality. Let us show that for any $\eta>
0$, there
exists~$K_2$ such that, provided $Z_t > K_2$,
\[
P[U_{t+1}=-b|U_t=-b,Z_t]>1-\eta.
\]
We have
\[
P[U_{t+1}=-b|U_t=-b,Z_t]= P[X_{t+1}
\in\mathcal{X}_2 | X_t \in\mathcal{X}_2,
Z_t].
\]
Again let us first work for a fixed $X_t \in\mathcal{X}_2$.
\begin{eqnarray*}
K_{\theta_t}(X_t, \mathcal{X}_2) &=& 1 -
K_{\theta_t}(X_t, \mathcal{X}_1)
\\
&=& 1 - \biggl[\int_{\mathcal{X}_1} Q(X_t, y)
\rho_{\theta_t}(X_t, y) \,dy \biggr]
\\
&=& 1 - \biggl[\int_{\mathcal{X}_1} Q(X_t, y) \biggl( 1
\wedge\frac{\pi(y)}{\pi(X_t)}\frac{Q(y,X_t)}{Q(X_t,y)}e^{-Z_t}
\biggr)\,dy \biggr].
\end{eqnarray*}
Using the assumption that the MH ratio
$\frac{\pi(y)}{\pi(x)}\frac{Q(y,x)}{Q(x,y)}$ is bounded from above, there
exists $K_2$ such that
\[
\forall k \geq K_2\ \forall x,y\in\mathcal{X}\qquad \frac{\pi(y)}{\pi(x)}
\frac{Q(y,x)}{Q(x,y)}e^{-k} \leq1.
\]
And hence for $Z_t > K_2$,
\begin{eqnarray*}
K_{\theta_t}(X_t, \mathcal{X}_2) &=& 1 -
e^{-Z_t}\int_{\mathcal{X}_1} Q(X_t, y)
\frac{\pi(y)}{\pi(X_t)}\frac{Q(y,X_t)}{Q(X_t,y)}\,dy
\\
& >& 1-e^{-Z_t}\int_{\mathcal{X}_1}Q(X_t, y) M \,dy
\\
& >& 1-Me^{-Z_t},
\end{eqnarray*}
and hence for any $\eta$, there is a $K_3$ greater than $K_2$ such that
for all
$Z_t \geq K_3$,
\[
K_{\theta_t}(X_t, \mathcal{X}_2) > 1 - \eta.
\]
We thus obtain
\[
P[U_{t+1}=-b|U_t=-b,Z_t] > 1 - \eta.
\]

To conclude we finally define $\epsilon= q_{\mathrm{min}} \mu(\mathcal{X}_2)$,
and then
for any $\eta> 0$, by taking any $\bar{Z}^{\mathrm{hi}}$ greater than $K_1
\vee
K_3$, we
have both inequalities.
\end{pf}

Considering the symmetry of the problem, we instantly have the following
corollary result. It states that if $Z_t$ goes too low, it has a strictly
positive probability of starting to increase, and when that happens, it keeps
on increasing with a high probability.
%
%
\begin{lemma}\label{lemma2}
With the introduced processes $Z_t$ and $U_t$,
there exists $\epsilon> 0$ such that for all $\eta> 0$, there
exists $\bar{Z}^{\mathrm{lo}}$ such that, if $Z_t \leq\bar{Z}^{\mathrm{lo}}$, we have the
following two inequalities:
\begin{eqnarray*}
P[U_{t+1}=+a|U_t=-b,Z_t]&>&\epsilon,
\\
P[U_{t+1}=+a|U_t=+a,Z_t]&>&1-\eta.
\end{eqnarray*}
\end{lemma}

\subsection{A new process that bounds $(Z_t)$ outside the set}

In this section, the proof introduces a new sequence of increments
$\tilde U_t$ that bounds
$U_t$, and such that the sequence $\tilde{Z}_t$ using $\tilde U_t$ as
increments,
\[
\tilde{Z}_{t+1} = \tilde{Z}_t + \tilde U_t
\]
returns to $[\bar{Z}^{\mathrm{lo}}, \bar{Z}^{\mathrm{hi}}]$ in a finite time whenever
it leaves
it. It will imply that $Z_t$ also returns to $[\bar{Z}^{\mathrm{lo}}, \bar
{Z}^{\mathrm{hi}}]$ in
finite time whenever it leaves it. Figure \ref{figZandZtilde} might
help to
visualize the proof.

%
\begin{figure}

\includegraphics{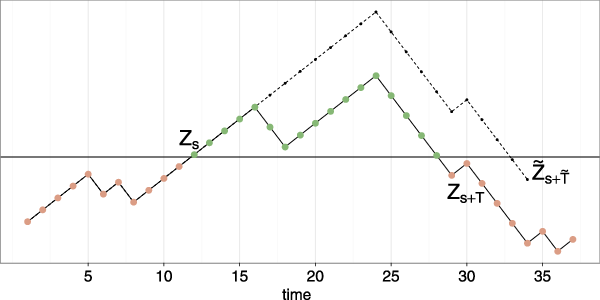}

\caption{Trajectory of $Z$ (full line with
dots) and
of $\tilde{Z}$ (dotted line), when these processes go above some level
$\bar Z^{\mathrm{hi}}$ indicated by a horizontal full line. $Z$ goes above the level
at time $s$, and returns below it at time $s + T$, whereas $\tilde{Z}$
stays above the level until time $s + \tilde T$, with $T\leq\tilde
T$.}\label{figZandZtilde}
\end{figure}

First let us use Lemma \ref{lemma1}. We can take $\epsilon<1/2$ and
$\eta<\min(1/2,\epsilon b/a)$. The lemma gives the existence of an
integer $K$
such that if $Z_t\geq K$, we have the following two inequalities:
%
%
\begin{eqnarray}
P[U_{t+1}=-b|U_t=+a,Z_t]&>&\epsilon,
\label{epsilon}
\\
P[U_{t+1}=-b|U_t=-b,Z_t]&>&1-\eta.\label{eta}
\end{eqnarray}

Suppose that there is some time $s$ such that $Z_{s-1} \leq K$ and
$Z_s\geq K$.
Note that necessarily $Z_s\in[K,K+a]$. Then we define $\tilde{Z}_s =
Z_s$, a
new process starting at time $s$. Let $s+T$ be the first time after $s$ such
that $Z_{s+T}\leq K$. We wish to show that $E[T]<\infty$.

We define the sequence of random variables $ (Z_t )_{t\geq s}$ defined
by $\tilde{Z}_s = Z_s$ and $\tilde{Z}_{t+1} = \tilde{Z}_t + \tilde U_t$
for $t>s$,
where $ (\tilde U_t )_{t\geq s}$ is a sequence of random
variables taking
the values $+a$ or $-b$.

For $s\leq t< T$, $\tilde U_t$ is defined as follows:
\begin{itemize}
\item if $U_{t+1}=+a$, then $\tilde U_{t+1}=+a$;
\item if $U_{t+1}=-b$, $U_t=-b$ and $\tilde U_t=-b$, then $\tilde
U_{t+1}=-b$ with probability
$p_1=(1-\eta)/P[U_{t+1}=-b|U_t=-b,Z_t]$ and $\tilde U_{t+1}=+a$ otherwise;
\item if $U_{t+1}=-b$, $U_t=+a$ and $\tilde U_t=+a$, then $\tilde
U_{t+1}=-b$ with
probability $p_2=\epsilon/P[U_{t+1}=-b|U_t=+a, Z_t]$ and $\tilde
U_{t+1}=+a$ otherwise;
\item if $U_{t+1}=-b$, $U_t=-b$ and $\tilde U_t=+a$, then $\tilde
U_{t+1}=-b$ with probability
$p_3=\epsilon(1+P[U_{t+1}=+a|U_t=-b,Z_t]/P[U_{t+1}=-b|U_t=-b,Z_t])$
and $\tilde U_{t+1}=+a$
otherwise.
\end{itemize}

For times $t\geq T$, $\tilde U_t$ is a Markov chain independent of
$U_t$ and $Z_t$,
with transition matrix
\[
\pmatrix{ 1-\epsilon& \epsilon
\vspace*{2pt}\cr
\eta& 1-\eta},
\]
where the first state corresponds to $+a$ and the second state to
$-b$.

First, let us check that all these probabilities are indeed less than 1.
For $p_1$, it follows from inequality (\ref{eta}).
For $p_2$, it follows from inequality (\ref{epsilon}).
For $p_3$, we have
\[
\epsilon\biggl(1+\frac
{P[U_{t+1}=+a|U_t=-b,Z_t]}{P[U_{t+1}=-b|U_t=-b,Z_t]} \biggr) \leq
\epsilon\biggl(1+
\frac{\eta}{1-\eta} \biggr)\leq2\epsilon\leq1,
\]
where we used the conditions $\eta<1/2$ and $\epsilon<1/2$. Hence
$(\tilde U_t)$ is
well defined.

%
\begin{lemma}\label{lemmautilde}
$(\tilde U_t)$ is a Markov chain over the space $\{+a,-b\}$ with
transition matrix
\[
\pmatrix{ 1-\epsilon& \epsilon
\vspace*{2pt}\cr
\eta& 1-\eta},
\]
where the first state corresponds to $\{+a\}$ and the second state to
\{$-b\}$.
\end{lemma}

\begin{pf}
We only need to check this for times $t\leq T$.
The events $\{\tilde U_t=-b\}$ and $\{\tilde U_t=-b,U_t=-b\}$ are
identical, hence
\begin{eqnarray*}
P[\tilde U_{t+1}=-b|\tilde U_t=-b, Z_t]&=&P[
\tilde U_{t+1}=-b|\tilde U_t=-b,U_t=-b,Z_t]
\\
&=&P[\tilde U_{t+1}=-b|U_{t+1}=-b,\tilde U_t=-b,U_t=-b,Z_t]
\\
& &{}\times P[U_{t+1}=-b|\tilde U_t=-b,U_t=-b,Z_t]
\\
&=&\frac{(1-\eta)P[U_{t+1}=-b|U_t=-b,Z_t]}{P[U_{t+1}=-b|U_t=-b,Z_t]}
\\
&=&1-\eta.
\end{eqnarray*}
Note that this does not depend on $Z_t$.

Similarly,
\begin{eqnarray*}
&&P[\tilde U_{t+1}=-b|\tilde U_t=+a, U_t=+a,Z_t]
\\
&&\qquad=P[\tilde U_{t+1}=-b|U_{t+1}=-b,\tilde U_t=+a,
U_t=+a,Z_t]
\\
&&\qquad\quad{} \times P[U_{t+1}=-b|\tilde U_t=+a,U_t=+a,Z_t]
\\
&&\qquad= \frac{\epsilon P[U_{t+1}=-b|U_t=+a,Z_t]}{P[U_{t+1}=-b|U_t=+a,
Z_t]}
\\
&&\qquad=\epsilon
\end{eqnarray*}
and
\begin{eqnarray*}
&& P[\tilde U_{t+1}=-b|\tilde U_t=+a, U_t=-b,Z_t]
\\
&&\qquad=P[\tilde U_{t+1}=-b|U_{t+1}=-b,\tilde U_t=+a,
U_t=-b,Z_t]
\\
&&\qquad\quad{} \times P[U_{t+1}=-b|\tilde U_t=+a,U_t=-b,Z_t]
\\
&&\qquad= \epsilon\biggl(1+\frac
{P[U_{t+1}=+a|U_t=-b,Z_t]}{P[U_{t+1}=-b|U_t=-b,Z_t]} \biggr)
\\
&&\qquad\quad{} \times P[U_{t+1}=-b|U_t=-b,Z_t]
\\
&&\qquad=\epsilon\bigl(P[U_{t+1}=-b|U_t=-b,Z_t] +
P[U_{t+1}=+a|U_t=-b,Z_t] \bigr)
\\
&&\qquad=\epsilon.
\end{eqnarray*}

These last two calculations result in
\[
P[\tilde U_{t+1}=-b|\tilde U_t=+a]=\epsilon
\]
with no dependence on $Z_t$ (or $U_t$).
\end{pf}

The previous lemma is central to the proof, and especially the lack of
dependence on $Z_t$. We always have $\tilde U_s = +a$, since $U_s =
+a$. Hence for
each $t\geq s$, the distribution of $\tilde U_t$ depends only on $\eta
$ and
$\epsilon$, and implicitly on the threshold $K$, but not on the value of
$Z_s$. Hence $(\tilde U_t)$ has the same law, every time the process
$(Z_t)$ goes
above $K$.

\subsection{\texorpdfstring{Conclusion: Proof of Theorem \protect\ref{thmconv} and Corollary \protect\ref{corolFH}}
{Conclusion: Proof of Theorem 1 and Corollary 2}}

Let us now use the bounding process $(\tilde{Z}_t)$ to control the time
spent by $(Z_t)$ above $K$.

%
\begin{lemma}\label{lemmafiniteexpect}
There exists $\tau\in\mathbb{R}$ such that, for all times $s$ such that
$Z_{s-1} \leq K$ and $Z_s \geq K$, and defining $T$ by $T = \inf_{d
\geq
0}\{Z_{s+d} \leq K\}$, then
\[
\mathbb{E}[T] \leq\tau.
\]
\end{lemma}

\begin{pf}
The Markov chain $(\tilde U_t)$ admits the following stationary distribution:
\[
\pi_{\tilde U} = \biggl(\frac{\eta}{\epsilon+\eta},\frac
{\epsilon
}{\epsilon+\eta} \biggr).
\]
Let us denote by $\tilde T$ the time spent by $(\tilde{Z}_t)$ over $K$,
that is,
\[
\tilde T = \inf_{d \geq0} \Biggl\{\sum
_{t=s+1}^ {s+d}\tilde U_t\leq-a \Biggr\}.
\]
Remember that $\tilde{Z}_s = Z_s \in[K, K+a]$, hence $\tilde
{Z}_{s+\tilde
T} \leq K$ (whatever the value of $Z_s$). Now, our choice of $\eta$
results in $a\eta<b\epsilon$ which implies $E[\tilde T]<\infty$
[\citet{Norris1998}]. Let $\tau=E[\tilde T]$. Note that since the law of
$(\tilde U_t)$ does not depend on the value of $Z_s$, $\tau$ does not
depend on
$Z_s$.

Since, for $t\leq T$, we impose that ``if $U_{t+1}=+a$, then $\tilde
U_{t+1}=+a$,''
it follows that $\forall t\leq T, U_t\leq\tilde U_t$. Consequently
$\forall
t\leq T, Z_t\leq\tilde Z_t$ and hence $T\leq\tilde T$. Note that (the
distribution of) $T$ depends on the exact value of $Z_s$, but that
$\tilde
T$ as we have defined it has a fixed distribution. We have $E[T]\leq
\tau$
(whatever the value $Z_s$).
\end{pf}

\begin{pf*}{Proof of Theorem \ref{thmconv}}
Let us define the following sequence
of indices:
\begin{eqnarray*}
S_1 &=& \inf_{s \geq0}\{Z_{s-1} \leq K \mbox{
and } Z_s \geq K\};\\
 S_k &=& \inf_{s \geq S_{k-1}}
\{Z_{s-1} \leq K \mbox{ and } Z_s \geq K\}.
\end{eqnarray*}
The sequence $(S_k)$ represents the times at which the process $(Z_t)$
goes above $K$.
Moreover let us introduce the sequence of time spent above $K$,
\[
T_k = \inf_{s \geq0}\{Z_{S_{k} + s - 1} \geq K \mbox{
and } Z_{S_{k} +
s} \leq K\}.
\]
We have $Z_{S_k}\in[K,K+a]$. Define $k(t)$ such that
$S_{k(t)} \leq t < S_{k(t)+1}$.
Either $Z_t \leq K$ or $Z_t > K$. In the latter case, $Z_t \leq
Z_{S_{k(t)}} + a T_{k(t)}$. Clearly, in any case,
%
%
\begin{equation}
\mathbb{E}[Z_t] \leq(K+a) + a\tau.\label{upperbound}
\end{equation}
A similar reasoning on the lower bound leads to $K'$ and $\tau
'<\infty$ such that
%
%
\begin{equation}
\mathbb{E}[Z_t] \geq\bigl(K'- b \bigr) - b
\tau'.\label{lowerbound}
\end{equation}
Inequalities (\ref{lowerbound}) and (\ref{upperbound}) imply
\[
\mathbb{E} \biggl[\frac{Z_t}{t} \biggr]\to0.
\]
\upqed\end{pf*}

As stated at the beginning of the section,
for update (\ref{eqrightupdate})
the convergence $Z_t/t\to0$ (in mean) implies
the convergence of the proportions $(\nu_t / t)$
to $\phi$ (also in mean).
We now show that this ensures that the flat
histogram is reached in finite time.
\begin{pf*}{Proof of Corollary \ref{corolFH}}
For a fixed threshold $c$, recall that FH being reached at time $t$
corresponds to the event
\[
\mathrm{FH}_t = \biggl\{\forall i \in\{1, \ldots, d\}\ \biggl
\llvert
\frac{\nu_t(i)}{t} - \phi_i \biggr\rrvert< c \biggr\}.
\]
We will only use the convergence in probability of the proportions
to $\phi$ for all $i$
\[
\frac{\nu_t(i)}{t} \mathop{\longrightarrow}^{\mathbb{P}}_{t\to\infty} \phi_i,
\]
which implies
\[
\forall\varepsilon> 0\ \exists N \in\mathbb{N}\ \forall t \geq N\qquad
\mathbb{P}(\mathrm{FH}_t) \geq1 - \varepsilon.
\]
We can hence define a stopping time $T^{\mathrm{FH}}$ corresponding
to the first FH being reached,
\[
T^{\mathrm{FH}} = \inf_{t \geq0}\{\mathrm{FH}_t\}
\]
and some $\varepsilon> 0$ such that
\[
\exists N \in\mathbb{N}\ \forall n \geq N \qquad \mathbb{P} \bigl
(T^{\mathrm
{FH}} \leq
N + n \bigr) \geq\varepsilon.
\]
Using Lemma 10.11 of \cite{williams}, the expectation of $T^{\mathrm
{FH}}$ is then finite.
\end{pf*}

\section{\texorpdfstring{Proof when $d\geq2$}{Proof when d >= 2}}\label{secproofanyd}

In this section we extend the proof to the more general case $d \geq2$.
Having proved that for $d = 2$, only update (\ref{eqrightupdate}) is
valid, we
now focus on this update and omit update (\ref{eqwrongupdate}).

We consider the log penalties defined for update (\ref{eqrightupdate}) by
\[
\log\theta_t(i) = \nu_t(i) (1 - \phi_i) -
\bigl(t - \nu_t(i) \bigr) \phi_i = \nu_t(i) -
t\phi_i,
\]
where $\nu_t(i)$ is the number of visits of $(X_t)$ in
$\mathcal{X}_i$. We assume without loss of generality that $\log
\theta
_0 = 0$. Then
$(X_t, \log\theta_t)$ is a Markov chain, by definition of the WL
algorithm. We first
prove that $(X_t, \log\theta_t)$ is $\lambda$-irreducible, for a
sigma-finite measure
$\lambda$. We will require the following additional assumption on
the desired frequencies $\phi$.
%
%
\begin{assumption}{\label{asrational}}
The desired frequencies are rational numbers,
\[
\phi= (\phi_1, \ldots, \phi_d) \in
\mathbb{Q}^d.
\]
\end{assumption}

%
\begin{lemma}\label{lemmairreducibility}
Let $\Theta$ be the following subset of $\mathbb{R}^d$:
\[
\Theta= \Biggl\{z\in\mathbb{R}^d \dvtx\exists(n_1,
\ldots, n_d) \in\mathbb{N}^d z_i =
n_i - \phi_i S_n \mbox{ where }
S_n = \sum_{j=1}^d
n_j \Biggr\}.
\]
Then denoting by $\lambda$ the product of the Lebesgue measure $\mu$ on
$\mathcal{X}$ and of the counting measure on $\Theta$, $(X_t, \log
\theta_t)$ is
$\lambda$-irreducible.
\end{lemma}

\begin{pf}
The proof essentially comes from B\'ezout's lemma, and is detailed in the
\hyperref[app]{Appendix}. Note, however, that it relies on Assumption
\ref{asrational}, that was not required for the case $d = 2$.
Although not a very satisfying assumption, which
is likely not to be necessary for proving the occurrence of FH in finite
time, it seems to be necessary for the irreducibility of $(X_t, \log
\theta_t)$,
at least with respect to a standard sigma-finite measure. In
any case, this assumption is not restrictive in practice.
\end{pf}

Since this chain is $\lambda$-irreducible, the proportion of visits to any
$\lambda$-measura\-ble set of $\mathcal{X}\times\Theta$ converges to a
limit in
$[0,1]$. This implies that the vector $(\nu_t(i)/t)$ converges to some vector
$(p_i)$. The following is a \textit{reductio ad absurdum}.

Suppose that for some $i\in\{1,\ldots,d\}$, $p_i\neq\phi_i$. Since
the vectors
$p$ and $\phi$ both sum to $1$, this means that for some $i$,
$p_i<\phi_i$:
such a state $i$ is visited less than the desired frequency.\vadjust{\goodbreak}

Let $\{i_1,i_2,\ldots\} = \operatorname{argmin}_{1\leq j\leq d}(p_j
-\phi_j)$.
Then for
any $i_k$ and for $j \notin\{i_1,\break i_2,\ldots\}$, we have
\[
Z_t^{j, i_k}=-\nu_t(i_k) +
\nu_t(j) + t(\phi_{i_k}-\phi_j) \sim
t(-p_{i_k}+\phi_{i_k}+p_j-\phi_j)\to
\infty.
\]
This implies
\[
\forall K>0\ \exists T \in\mathbb{N}\ \forall t>T\qquad  Z_t^{j,i_k}>K.
\]

Now consider the stochastic process $(U_t)$ such that:
\begin{itemize}
\item$U_t=-b$ if $J(X_t)\in\{i_1,i_2,\ldots\}$;
\item$U_t=+a$ otherwise,
\end{itemize}
for some real numbers $a$ and $b$. Recall that the function $J$ is such
that if
$X_t\in\mathcal X_i$, then $J(X_t)=i$.

Let $\epsilon$ be such that when
$X_t\notin\mathcal{X}_{i_1}\cup\mathcal{X}_{i_2}\cup\cdots$,
there is
probability at least $\epsilon$ of proposing in
$\mathcal{X}_{i_1}\cup\mathcal{X}_{i_2}\cup\cdots$. For large enough
$K$, these
proposals will always be accepted. As before, for large enough $K$, we can
make the probability $\eta$ of leaving
$\mathcal{X}_{i_1}\cup\mathcal{X}_{i_2}\cup\cdots$ as small as we wish.

Using the exact same reasoning as in Section \ref{secdequals2}, we can
construct a process $(\tilde U_t)$
which is a Markov chain with transition matrix
\[
\pmatrix{ 1-\epsilon& \epsilon
\vspace*{2pt}\cr
\eta& 1-\eta}
\]
and with $U_t < \tilde U_t$ almost surely. Therefore for $t>T$, $(U_t)$
decreases on
average, hence $(Z_t^{j,i_k})$ decreases on average, which contradicts the
assumption that it goes to infinity. Hence for all $i$, $p_i = \phi_i$.

\section{\texorpdfstring{Illustration of Theorem \protect\ref{thmconv} on a toy example}
{Illustration of Theorem 1 on a toy example}}\label{sectoyexample}

Let us show the consequences of Theorem \ref{thmconv} on a simple
example. We
consider as the target distribution the standard normal distribution truncated
to the set $\mathcal X=[-10,10]$. We use a Gaussian random walk
proposal, with
unit standard deviation. Finally we arbitrarily split the state space in
$\mathcal{X}_1 = [-10,0]$ and $\mathcal{X}_2 = ]0,10]$, and we set
the desired
frequencies to be $\phi= (0.75, 0.25)$. Figure \ref{figtoyexample}
shows the
results of the Wang--Landau algorithm. Using update (\ref
{eqrightupdate}) and
200,000 iterations, we obtain the histogram of Figure \ref{figtoyexample}(a).
Figure \ref{figtoyexample}(b) shows the convergence of the
proportions of visits to each bin, using update (\ref{eqrightupdate}). The
dotted horizontal lines indicate $\phi$, and we can check that the observed
proportions of visits converge toward it.

Figure \ref{figtoyexample}(c) shows a similar plot, this
time using
update (\ref{eqwrongupdate}). Again, the desired frequencies are represented
by dotted lines. Using the left-hand side of equation
(\ref{equpdateconstraint}), we can calculate the theoretical limit of the
observed proportion of visits in each bin, which for $\gamma= 1$ and
$\phi=
(0.75, 0.25)$, is approximately equal to $(0.79, 0.21)$. Hence for a precision
threshold $c$ equal to, for example, $1\%$, the occurrence of FH is
not likely to
occur if one uses update~(\ref{eqwrongupdate}).

As expected, update (\ref{eqrightupdate}) leads to convergence to the desired
frequencies, but update (\ref{eqwrongupdate}) does not.

%
\begin{figure}
\centering
\begin{tabular}{@{\hspace{-5pt}}c@{\hspace*{4pt}}c@{\hspace*{4pt}}c@{}}

\includegraphics[scale=0.95]{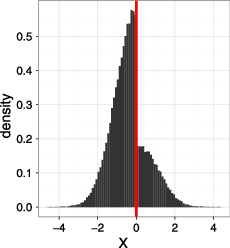}
 & \includegraphics[scale=0.95]{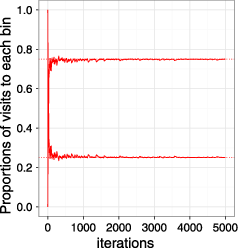} & \includegraphics[scale=0.95]{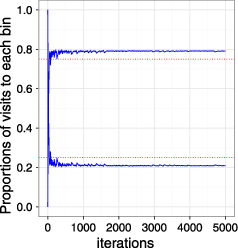}\\
\footnotesize{(a) Histogram of the} & \footnotesize{(b) Convergence
of the} & \footnotesize{(c) Convergence of the}\\
\footnotesize{generated sample} & \footnotesize{proportions
of visits to each bin,} & \footnotesize{proportions of visits to each
bin}\\
 & \footnotesize{using the right
update} & \footnotesize{using
the wrong update}
\end{tabular}
\caption{Results of the Wang--Landau algorithm using
two different updates of the penalties. Histogram of the generated sample
using update (\protect\ref{eqrightupdate}), with a vertical line
showing the
binning (left). Convergence of the proportions of visits to each bin,
using update (\protect\ref{eqrightupdate}) (middle) and using update
(\protect\ref{eqwrongupdate}) (right). The dotted horizontal lines
represent the
desired frequencies.}\label{figtoyexample}
\end{figure}

%

\section{Discussion}\label{secdiscussion}

As seen in Theorem \ref{thmconv} and Corollary \ref{corolFH} of
Section \ref{secdequals2}, update~(\ref{eqrightupdate}) is valid, in the sense that the frequencies of
visits of
the chain $(X_t)$ converges toward $\phi$. Consequently FH is met
in finite time, for any threshold $c > 0$.

Regarding the proof of Theorem \ref{thmconv} in the case $d > 2$, we assume
that the desired frequencies $\phi$ are rationals (Assumption \ref
{asrational}), which allows to prove that
the Markov chain generated by the algorithm $(X_t, Z_t)$ is
$\lambda$-irreducible for some sigma-finite measure $\lambda$.
However, our
proof requires mainly that the proportions of visits of $(X_t)$ to any bin
$\mathcal{X}_i$ converge, which is equivalent to the convergence of
$(Z_t /
t)$. We believe that results on random walks in random environments
[\citet{zeitouni2006random}] would allow us to remove the rationality
assumption.

Assumptions \ref{ascompact}--\ref{asmhratio} could be relaxed by using
the well-known properties of the Metropolis--Hastings algorithm, from
which we
did not take advantage here. More precisely, note that the Wang--Landau
transition kernel differs from the Metropolis--Hastings only when the proposed
points, generated through $Q(\cdot,\cdot)$, land in a different bin
than the
current position of the chain. Otherwise, the kernel behaves like a
Metropolis--Hastings targeting $\pi$. Hence under some weaker
assumptions than
the one we have formulated here, it has recurrence properties.

To conclude, we have shown that for fixed $\gamma$, the Flat Histogram
criterion is reached in finite time for certain updates. For other
updates, the
observed frequencies do not converge to the desired frequencies, and so there
is a nonzero probability that the flat histogram criterion will never be
verified. Note that we do not make any claims about the distribution of the
sample inside each of the bins $\mathcal{X}_i$ at fixed $\gamma$.

\begin{appendix}\label{app}

\section*{\texorpdfstring{Appendix: Proof of Lemma \lowercase{\protect\ref{lemmairreducibility}}}
{Appendix: Proof of Lemma 7}}

Let $\Theta\subset\mathbb{R}^d$ be the set of possibly reachable
values of
the process $(\log\theta_t)$. We define it by
\[
\Theta= \Biggl\{z\in\mathbb{R}^d \dvtx\exists(n_1,
\ldots, n_d) \in\mathbb{N}^d z_i =
n_i - \phi_i S_n \mbox{ where }
S_n = \sum_{j=1}^d
n_j \Biggr\}.
\]
We want to prove the existence of a measure $\lambda$ on
$\mathcal{X}\times\Theta$ such that the Markov chain $(X_t, \log
\theta
_t)$ is
$\lambda$-irreducible. Denote by $\mu$ the Lebesgue measure on
$\mathcal{X}$,
and let $A\in\mathcal{B}(\mathcal{X})$ such that $\mu(A) > 0$, and let
$z^\star
\in\Theta$. Let us show that for any time $s$ at which $X_s = x_s
\in\mathcal{X}$ and $\log\theta_s = z_s\in\Theta$, there exists
$t > 0$
such that
$X_{s+t} \in A$ and $\log\theta_{s+t} = z^\star$ with strictly positive
probability.
This will prove the $\lambda$-irreducibility of $(X_t, \log\theta_t)$
where $\lambda$ is
the product of the Lebesgue measure $\mu$ on $\mathcal{X}$ and the counting
measure on $\Theta$.

Note first that for any $n = (n_1, \ldots, n_d) \in\mathbb{N}^d$,
the process
$(X_t)$ can visit exactly $n_i$ times each set $\mathcal{X}_i$ (for
all $i$)
between some time $s+1$ and some time $s + \sum_{i=1}^d n_i$, since
there is
always a nonzero probability of $X_{t+1}$ visiting any $\mathcal{X}_i$ given
$X_t$ and $\log\theta_t$ (using Assumptions \ref{asproposal} on the proposal
distribution and the form of the MH kernel). More formally, given any
$n\in
\mathbb{N}^d$ and any time $s$, denoting $S_n = \sum_{i=1}^d n_i$,
%
%
\begin{equation}
\label{eqpossiblepath} \mathbb{P} \Biggl( \forall i \in\{1, \ldots, d\} \sum
_{t =
s+1}^{s+S_n} \Ind_{\mathcal{X}_i}(X_{t})
= n_i \bigg\vert X_s, \log\theta_s \Biggr) >
0.
\end{equation}

Furthermore since $\mu(A) > 0$ and since $(\mathcal{X}_i)_{i=1}^d$ is a
partition of $\mathcal{X}$ (satisfying Assumption \ref{asemptybins} on
nonempty bins), there exists $B\subset A$ such that $B\subset
\mathcal{X}_{i^\star}$ for some $i^\star\in\{1, \ldots, d\}$ and
$\mu
(B) > 0$.
We are going to prove the following statement, which means that there
is a
``path'' between any pair of points in $\Theta$:
%
%
\begin{lemma}\label{eqsubirreducibility}
\[
\forall z^{1}, z^{2} \in\Theta\ \exists n\in
\mathbb{N}^d\ \forall i \in\{1, \ldots, d\}\qquad z_i^{1}
+ n_i - \Biggl(\sum_{j=1}^d
n_j \Biggr) \phi_i = z_i^{2}.
\]
\end{lemma}
Then we will conclude as follows: the Markov chain can go from any
$(x_s, z_s)$
to some $(x_{s + t-1}, z_{s + t-1})$ where $z_{s + t-1}$ can be
anywhere in
$\Theta$, and then in one final step to $(x_{s+t}, z_{s+t})$ such that
$x_{s+t}\in B$ and $z_{s+t} = z^\star$, since $z_{s + t - 1}$ can be chosen
such that $z_{s+t} = z^\star$ when $x_{s+t}\in B\subset\mathcal
{X}_{i^\star}$.

\begin{pf}
The structure of the proof is the
following: we prove that $(\log\theta_t)$ can go from $0$ to $0$, then
from any
$z\in\Theta$ to $0$, and the possibility of going from $0$ to any $z
\in
\Theta$ comes from the definition of $\Theta$.

Suppose that $\log\theta_0= (0, \ldots, 0)$, and let us prove that
the process
$(\log\theta_t)$ can go back to $0$, that is, let us find a vector $n
= (n_1, \ldots, n_2)
\in\mathbb{N}^d$ such that
\[
\forall i \in\{1, \ldots, d\}\qquad 0 = n_i - \phi_i
S_n \qquad\mbox{where } S_n =\sum
_{j=1}^d n_j.
\]
Under the rationality assumption on $\phi$ (Assumption \ref
{asrational}), there exists $(a_i, \ldots,\break
a_d)\in\mathbb{N}^d$ and $b\in\mathbb{N}$ such that $\phi_i = a_i
/ b$ for
all $i$. Now define $n\in\mathbb{N}^d$ as follows:
\[
\forall i\in\{1,\ldots, d\}\qquad n_i = k \prod
_{j= 1, j\neq i}^d \frac{1}{a_j},
\]
where $k\in\mathbb{N}$ is such that $n_i\in\mathbb{N}$ for all $i$. Then
using $\sum_{j=1}^d \phi_j = 1$ one can readily check that
\[
\forall i\in\{1,\ldots, d\} \qquad n_i - \phi_i \Biggl( \sum
_{j=1}^d n_j \Biggr) = 0.
\]
Hence the vector $n$ defines a possible path for $(\log\theta_t)$
between $0$ and
$0$, in $S_n = \sum_{j=1}^d n_j$ steps, with a strictly positive
probability [using equation (\ref{eqpossiblepath})].

A similar reasoning allows us to find a possible path from any $z\in
\Theta$ to
$0$. For such a $z\in\Theta$, there exists $(m_1,\ldots,m_d)\in
\mathbb{N}^d$ such that
%
%
\begin{equation}
\label{defm} \forall i\in\{1,\ldots, d\}\qquad z_i = m_i -
S_m a_i/b \qquad \mbox{where } S_m =\sum
_{j=1}^d m_j.
\end{equation}
We wish to show that there exits $(k_1,\ldots,k_d)\in\mathbb{N}^d$ such
that $k_i - S_k a_i/b = -z_i$ for all $i$, where $S_k=\sum_{j=1}^d k_j$.
To construct $(k_1,\ldots,k_d)$, we use the already introduced vector
$(n_1,\ldots,n_d)$ such that $n_i - S_n a_i/b = 0$ for all $i$, where
$S_n=\sum_{j=1}^d n_j$. Putting this together with (\ref{defm}), we get
for any $C\in\mathbb{N}$,
%
%
\begin{equation}
-z_i + C*0 = -m_i +C*n_i -
\frac{a_i}{b}(CS_n -S_m).
\end{equation}

For $C$ large enough, for all $i$, $Cn_i -m_i > 0$. We simply take
$k_i = C
n_i-m_i$ for all $i$. This proves that starting from a point $z \in
\Theta$ (by definition reachable from $0$), $(\log\theta_t)$ can reach
$0$ again.
\end{pf}
\end{appendix}


\section*{Acknowledgments}
The authors thank Luke Bornn, Arnaud Doucet, Anthony Lee, \'Eric
Moulines, Christian P.
Robert and an anonymous reviewer for helpful comments.
Both authors contributed equally to this work.

%

%


\printaddresses

\end{document}